\documentclass{amsart}
\usepackage[fontsize=12pt]{fontsize}
\usepackage{amssymb,fix-cm}
\usepackage{enumitem}
\usepackage{tikz-cd}
\usepackage[a4paper,left=2.9cm,right=2.9cm,top=2.6cm,bottom=2.6cm]{geometry}
\usepackage[
  colorlinks=true,
  pdfstartview=FitH
]{hyperref}

\numberwithin{equation}{section}

\theoremstyle{plain}
\newtheorem{theorem}{Theorem}[section]
\newtheorem{lemma}[theorem]{Lemma}
\newtheorem{proposition}[theorem]{Proposition}
\newtheorem{corollary}[theorem]{Corollary}
\newtheorem{conjecture}[theorem]{Conjecture}

\theoremstyle{definition}
\newtheorem{definition}[theorem]{Definition}

\theoremstyle{definition}
\newtheorem{remark}[theorem]{Remark}
\newtheorem{example}[theorem]{Example}

\title[Sharp effective Nullstellensatz]{Sharp effective Nullstellensatz for projective Cohen-Macaulay schemes}
\author[Zhi Li]{Zhi Li}
\address{Zhi Li: School of Mathematical Sciences and Key Laboratory of Mathematics and Information Networks (Ministry of Education), Beijing University of Posts and Telecommunications, Beijing 100876, China.}
\email{lizhi@amss.ac.cn}

\author[Xiankui Meng]{Xiankui Meng}
\address{Xiankui Meng: School of Mathematical Sciences and Key Laboratory of Mathematics and Information Networks (Ministry of Education), Beijing University of Posts and Telecommunications, Beijing 100876, China.}
\email{mengxiankui@amss.ac.cn}

\author[Xiaoyu Su]{Xiaoyu Su}
\address{Xiaoyu Su: School of Mathematical Sciences and Key Laboratory of Mathematics and Information Networks (Ministry of Education), Beijing University of Posts and Telecommunications, Beijing 100876, China.}
\email{xiaoyusu@bupt.edu.cn}

\author[Langfeng Zhu]{Langfeng Zhu}
\address{Langfeng Zhu: School of Mathematics and Statistics, Wuhan University, Wuhan 430072, China}
\email{zhulangfeng@amss.ac.cn}


\begin{document}
\begin{abstract}
In this paper, we prove a sharp effective Nullstellensatz for projective Cohen-Macaulay schemes of pure dimension over an arbitrary field. In particular, if $L$ is an ample line bundle on a smooth complex projective variety $X$ and $\mathcal J\subseteq\mathcal O_X$ is a coherent ideal sheaf with $\mathcal J\otimes L$ globally generated, then $(\sqrt{\mathcal J})^{\deg_L X}\subseteq\mathcal J$. This establishes the sharp bound conjectured by Ein and Lazarsfeld.
More generally, we obtain sharp bounds in terms of mixed intersection numbers and a refinement involving products of powers of prime ideal sheaves associated to the irreducible components of the base locus. We also establish sharp radical power bounds for homogeneous ideals in finitely generated standard graded Cohen-Macaulay domains, extending Koll\'ar's sharp effective Nullstellensatz for polynomial rings. The proofs combine purity and cycle preservation for leading Ext sheaves with Koszul evaluation maps compatible with successive divisor cuts.
\end{abstract}
\maketitle

\section{Introduction}
For an ideal $I\subseteq \mathbb{C}[x_1,\dots,x_n]$, Hilbert's Nullstellensatz states that the vanishing ideal of the zero locus defined by $I$ equals the radical of $I$. In other words, if a polynomial $f$ vanishes on the zero locus of $I$, then $f\in\sqrt I$, which means $f^m\in I$ for some sufficiently large $m$. The effective version of the Nullstellensatz makes Hilbert's Nullstellensatz quantitative by giving explicit upper bounds for the exponent $m$. This question was first considered by Hermann \cite{Hermann26}, and a major breakthrough was achieved by Brownawell \cite{Brownawell87}. 
For polynomial rings, sharp bounds were obtained by Koll\'ar \cite{Kollar88} and Jelonek \cite{Jelonek05} over an arbitrary field, including bounds for powers of homogeneous radicals as well. In \cite{EinLazarsfeld}, Ein and Lazarsfeld formulated an effective Nullstellensatz in the setting of smooth complex projective varieties. 
For further developments and variants of the effective Nullstellensatz, we refer the reader to \cite{Skoda72, BrianconSkoda74, Sombra97, Brownawell98, Sombra99, KollarArbitrary, AW2015}.

This work is motivated by the results of Koll\'ar \cite{Kollar88} and  Ein-Lazarsfeld \cite{EinLazarsfeld}. Koll\'ar proved sharp effective Nullstellensatz bounds for polynomial rings over arbitrary fields, including estimates for ordinary powers of homogeneous radicals. In particular, on $\mathbb{P}^n$ with $n\geq 2$, for ideals generated by homogeneous polynomials of common degree $a\geq 3$, his result yields the exponent $a^n$. Subsequently, Ein and Lazarsfeld considered the effective Nullstellensatz bounds in the setting of smooth complex projective varieties: for an ample line bundle $L$ over an $n$-dimensional smooth projective variety $X$, and for an ideal sheaf $\mathcal{J}$ with $\mathcal{J} \otimes L$ globally generated, they established
\[
(\sqrt{\mathcal{J}})^{n\deg_L X}\subseteq\mathcal{J}.
\]
However, they suggested that the exponent $n\deg_L X$ may be sharpened to $\deg_L X$.

\begin{conjecture}[Ein-Lazarsfeld \cite{EinLazarsfeld}]\label{conj}
If $X$ is a smooth complex projective variety, $L$ is an ample line bundle, $\mathcal J$ is an ideal sheaf with $\mathcal{J}\otimes L$ globally generated, then
\[
(\sqrt{\mathcal{J}})^{\deg_L X}\subseteq\mathcal{J}.
\]
\end{conjecture}

The very ample case of this conjecture was proved by Jelonek \cite{Jelonek05}.

In this paper, we prove Conjecture \ref{conj} in full generality by establishing the following theorem in a general framework.

\begin{theorem}
\label{thm:unified-intersection}
Let $X$ be a projective Cohen-Macaulay scheme of pure dimension $n$ over a field $k$, and $\mathcal J\subseteq\mathcal O_X$ a coherent ideal sheaf.
Let \(B = V(\mathcal{J})\) denote the scheme defined by $\mathcal J$. For $0\leq r\leq n$, let $A, L_i$ be ample line bundles on $X$, and $s_i\in H^0(X,\mathcal J\otimes L_i)$ for $i=1,2,\ldots,r$.
Write $B_0=X$ and $B_j=V(s_1,\ldots,s_j)$ for $1\leq j\leq r$.
Assume that
\begin{enumerate}[label=\textup{(\arabic*)}]
\item $\dim(B_j\setminus|B|)\leq n-j$ for $1\leq j\leq r$;
\item $B_r\setminus|B|=\varnothing$ if $r<n$.
\end{enumerate}
Set $ d=\deg \left(c_1(L_1)\cdots c_1(L_r)c_1(A)^{n-r}\cap[X]\right)$.
Then $$(\sqrt{\mathcal J})^d\subseteq\mathcal J.$$

Moreover, there exist finitely many irreducible closed subsets
$W_\alpha\subseteq|B|$, endowed with their reduced subscheme structures,
and positive integers $m_\alpha$, such that $\bigcup_\alpha|W_\alpha|=|B|$,
\[
 \prod_\alpha\mathcal I_{W_\alpha}^{m_\alpha}\subseteq\mathcal J,
\]
and $\sum_\alpha m_\alpha  \deg \left( \left(\prod_{i=\mathrm {codim}_X W_\alpha+1}^{r}c_1(L_i)\right)    c_1(A)^{n-r}\cap[W_\alpha]  \right)\leq d$. 
Here $\mathrm {codim}_X W_\alpha\in\{0,\ldots,r\}$ and $\mathcal I_{W_\alpha}$ denotes the ideal sheaf of $W_\alpha$ in $X$.
\end{theorem}

As an application of Theorem~\ref{thm:unified-intersection}, we can obtain the following consequence by choosing suitable linear combinations of the generating sections.

\begin{corollary}
\label{thm:fewer-generators}
Let $X$ be a projective Cohen-Macaulay scheme of pure dimension $n$ over a field $k$. Let $L$ be an ample line bundle such that $\mathcal J\otimes L$ is globally generated by $m$ sections.
For any ample line bundle $A$, one has $(\sqrt{\mathcal J})^{d} \subseteq\mathcal J$ for
\begin{equation*}
d = 
\begin{cases}
   \deg (c_1(L)^m c_1(A)^{n-m}\cap[X]), & m<n; \\
   \deg (c_1(L)^n\cap[X]), & m\geq n.
\end{cases}
\end{equation*}
In particular, if $L=A^a$ with $a\geq1$, then $(\sqrt{\mathcal J})^{\deg_A X \cdot a^{\min\{m,n\}}}\subseteq\mathcal J$.
\end{corollary}

By taking $A=L$, Corollary \ref{thm:fewer-generators} yields the following result with a sharp bound, which gives an affirmative answer to Conjecture \ref{conj}.

\begin{corollary}
\label{thm:degree}
Let $X$ be a projective Cohen-Macaulay scheme of pure dimension $n$ over a field $k$. If $L$ is ample and $\mathcal J\otimes L$ is globally generated, then
\[
 (\sqrt{\mathcal J})^{\deg_LX}\subseteq\mathcal J.
\]
\end{corollary}

Applying Theorem~\ref{thm:unified-intersection} with suitable powers of an ample line bundle $L$, we obtain the following bounds in terms of the powers of $L$ and the degree of the variety with respect to $L$.

\begin{corollary}
\label{thm:mixed-degrees}
Let $X$ be an integral projective Cohen-Macaulay scheme of dimension $n$ over a field $k$. Let $L$ be an ample and globally generated line bundle, and let $f_i\in H^0(X,L^{a_i})$ be nonzero sections with $a_1\geq\cdots\geq a_m\geq1$, which generate $\mathcal J$.
Then one has $(\sqrt{\mathcal J})^{d} \subseteq\mathcal J$ for
\begin{equation*}
d = 
\begin{cases}
   \deg_L X \cdot a_1\cdots a_m, & m \leq n; \\
   \deg_L X \cdot a_1\cdots a_{n-1}a_m, & m>n>1; \\
   \deg_L X \cdot a_m, & n=1.
\end{cases}
\end{equation*}
\end{corollary}

The exponents in the preceding three corollaries are sharp.
Indeed, Example~\ref{ex:ambient-sharp} realizes the bound in Corollary~\ref{thm:mixed-degrees}, taking all degrees equal shows that the bound in Corollary~\ref{thm:fewer-generators} is sharp, and taking them all equal to one attains the exponent $\deg_L X$ in Corollary~\ref{thm:degree}. We also note that the Cohen-Macaulay hypothesis is essential for these sharp bounds, see Example~\ref{ex:cm-hypothesis}.

For $X=\mathbb P^n$, $L=\mathcal O_{\mathbb P^n}(1)$, $n\geq2$, and $a_i\geq3$, the exponent in Corollary~\ref{thm:mixed-degrees} agrees with Koll\'ar's bound for homogeneous ideals \cite[Theorem~1.5]{Kollar88}.
Koll\'ar's conclusion also controls the original homogeneous ideal at the vertex of the affine cone.
For $n=1$ and $m>1$, the respective exponents are $a_m$ for the ideal-sheaf assertion and $a_1+a_m-1$ for the homogeneous ideal assertion.

On a projective scheme, containment of ideal sheaves does not alone imply containment in the original homogeneous ideal: one must also control the contribution at the vertex of the affine cone. Our second main theorem supplies this additional, control which gives the optimal exponent under the Cohen-Macaulay hypothesis.

\begin{theorem}
\label{thm:graded-homogeneous-containment}
Let $S=\bigoplus_{\nu=0}^{\infty}S_\nu$ be a finitely generated standard graded Cohen-Macaulay domain over a field $k$, with $S_0=k$. 
Denote $Y=\operatorname{Proj}S$ and $n=\dim S-1\geq1$. Let $I=(f_1,\ldots,f_m)$, where the $f_i$ are homogeneous elements of degrees
$a_1\geq\cdots\geq a_m\geq1$. Set
\[d=\max\left\{
 \deg_{\mathcal O(1)} Y\cdot a_m a_1 \cdots a_{\min\{m,n\}-1},\
 \deg_{\mathcal O(1)} Y-1+a_m+\sum_{i=1}^{\min\{m-1,n\}}(a_i-1)
 \right\}.\]
Then $(\sqrt I)^d\subseteq I$.
\end{theorem}

For $S=k[x_0,\ldots,x_n]$, Theorem~\ref{thm:graded-homogeneous-containment} recovers Koll\'ar's bound \cite[Theorem 1.5]{Kollar88} without the assumption that every generator has degree at least $3$. Jelonek~\cite{Jelonek05} establishes effective bounds for homogeneous ideal containment without requiring the homogeneous coordinate ring to be Cohen-Macaulay. Jelonek's bound is optimal when $m \leq n$, whereas for $m>n$, Koll\'ar's theorem \cite[Theorem 1.5]{Kollar88} permits a slightly smaller exponent. The Cohen-Macaulay setting also appears in Sombra's work \cite[Theorem~1.8]{Sombra99}. 

The first term in the expression for $d$ in Theorem~\ref{thm:graded-homogeneous-containment} is the bound that already controls the projective ideal-sheaf inclusion. The second term accounts for the additional graded obstruction that corresponds to the term $a_1 + a_m - 1$ in \cite[Theorem 1.5]{Kollar88}. Both terms are necessary as Example~\ref{ex:graded-radical-sharpness} shows.

By homogenizing the polynomial functions on affine varieties, we obtain the following affine corollary from the preceding projective and graded bounds.

\begin{corollary}
\label{cor:graded-affine}
Let $V\subseteq\mathbf A^N_k$ be an integral affine variety of dimension $n\geq1$, and let $Y\subseteq\mathbb P^N_k$ be its projective closure.
Assume that the homogeneous coordinate ring $S$ of $Y$ is Cohen–Macaulay.
Let $f_1,\ldots,f_m\in k[V]$ be nonconstant polynomial functions of degrees $a_1\geq\cdots\geq a_m\geq1$, and $\mathfrak a=(f_1,\ldots,f_m)$. 
Then $$(\sqrt{\mathfrak a})^{\deg V \cdot a_m a_1\cdots a_{\min\{m,n\}-1}}\subseteq\mathfrak a.$$ 
If $\mathfrak a=k[V]$, then there exist $g_i\in k[V]$ satisfying
\begin{equation}\label{eq:graded-affine-certificate}
 1=\sum_{i=1}^m g_if_i,  \quad \deg g_i+a_i\leq d,
\end{equation}
where
\[
d=\max\left\{
\deg V\cdot a_m a_1 \cdots a_{\min\{m,n\}-1},
\ 
\deg V-1+a_m+\sum_{i=1}^{\min\{m-1,n\}}(a_i-1)
\right\}.
\]
In particular, the conclusion holds if $k$ is algebraically closed and the $f_i$ have no common zero.
\end{corollary}

The degree of a function in $k[V]$ is the smallest total degree of any polynomial representing it on $V$, relative to the given embedding $V\subseteq\mathbf A^N_k$. The estimate \eqref{eq:graded-affine-certificate} controls the degree of each $g_{i}$ and implies $\deg(g_if_i)\leq d$. On a general affine variety, $\deg(g_if_i)$ and $\deg g_i+a_i$ are different in general. For example, in $k[x,t]/(xt^{\delta-1}-1)$ with $\delta\geq2$, the functions $f=x^a$ and $g=t^{(\delta-1)a}$ satisfy $\deg f+\deg g = a \cdot \deg Y$ but $fg=1$. Thus the product has degree \(0\), while the sum of the degrees can be large.

Our proofs of Theorem~\ref{thm:unified-intersection} and Theorem~\ref{thm:graded-homogeneous-containment} use a homological approach based on duality and Koszul evaluation maps. This differs from the use of the Skoda-type division theorem in Ein-Lazarsfeld \cite{EinLazarsfeld}, and from the elimination methods of Jelonek \cite{Jelonek05}. 
We develop the successive cutting and degree-counting strategy of \cite{Kollar88} by combining pure leading Ext sheaves with compatible Koszul evaluation images.
The purity and cycle properties of the leading Ext sheaves are classical facts \cite{Schenzel,Stacks}. Their role here is as follows: at each step of removing $|B|$-torsion, they convert the degree of the removed top-dimensional cycle into an annihilator, which can then be transported to the next cut.

Section~\ref{sec:2} introduces the cycles, duality, and evaluation-images. Section~\ref{sec:all-dimensions} proves Theorem~\ref{thm:unified-intersection}.
Section~\ref{sec:refined-degrees} presents the proofs of Corollary~\ref{thm:fewer-generators} and Corollary \ref{thm:mixed-degrees}, and gives sharpness examples. 
Finally, Section~\ref{sec:graded-sharp-bounds} proves Theorem~\ref{thm:graded-homogeneous-containment} and Corollary~\ref{cor:graded-affine}, and provides sharpness examples.

\section{Cycles, leading Ext, and evaluation images}
\label{sec:2} This section collects some facts used in the proof of the main theorems. We first recall associated cycles and mixed intersection degrees. Using generic lengths, we obtain bounds for annihilators of pure sheaves.
Then we apply local Grothendieck duality with the dualizing sheaf to establish purity and cycle preservation for leading Ext sheaves on Cohen-Macaulay schemes.
Finally, we introduce the notion of the evaluation image ideal. Using Koszul complexes and relative duality maps, we construct a morphism from the leading Ext sheaves to the structure sheaf $\mathcal O_B$ of the base scheme, whose image is precisely the evaluation-image ideal. This construction allows our inductive argument to proceed, and translates annihilation statements for Ext sheaves into ideal-power containments in $\mathcal O_B$.

\subsection{Cycles and mixed intersection degrees}

We now recall the cycle-theoretic formalism which provides the quantitative degree estimates underlying the inductive bound. We begin with the cycle associated to a coherent sheaf, the notion of purity, and the intersection identities used throughout the proof.

\begin{definition}[\cite{FultonIT} Example 15.1.8, \cite{Stacks} Tag 02QX]
\label{def:hilbert-coefficients}
Let $X$ be a projective scheme over a field $k$, and let $\mathcal M$ be a coherent sheaf with $\dim\operatorname{Supp}\mathcal M\leq q$.
Its associated $q$-cycle is
\[
 [\mathcal M]_q :=  \sum_{\dim\overline{\{\eta\}}=q} \operatorname{Length}_{\mathcal O_{X,\eta}}(\mathcal M_\eta)
 [\overline{\{\eta\}}].
\]
The sum runs over the generic points of the $q$-dimensional irreducible components of $\operatorname{Supp}\mathcal M$, each endowed with its reduced structure. 
For a closed subscheme $Z\subseteq X$, we write $[Z]_q=[\mathcal O_Z]_q$. We use the same notation for these cycles and their classes in the Chow group $\operatorname{CH}_q(X)$.
\end{definition}

$[\mathcal M]_q$ is effective, and it is zero if $\dim\operatorname{Supp}\mathcal M<q$. It follows from \cite[Proposition 2.5]{FultonIT} that, for an invertible sheaf $L$, the intersection with its first Chern class defines an operator $ c_1(L)\cap-:
 \operatorname{CH}_q(X)\rightarrow
 \operatorname{CH}_{q-1}(X)$.
These operators commute and are additive with respect to tensor products of invertible sheaves. 
We use their products to form intersection numbers
\[
 \deg\bigl(  c_1(L_1)\cdots c_1(L_q)\cap[\mathcal M]_q \bigr).
\]
Here $\deg$ denotes the degree of a zero-cycle, with $\deg[x]=[\kappa(x):k]$ for a closed point $x$.
When $q=0$, the product of Chern classes is understood to be empty.
These constructions apply to singular and nonreduced schemes, see~\cite[Chapters~1--2]{FultonIT}.
The usual degree of $X$ with respect to an ample line bundle $L$ is
$\deg_L X = \deg\bigl(c_1(L)^{\dim X}\cap[X]_{\dim X}\bigr)$.

We collect some basic properties of cycles and intersection numbers, which form the basis for the degree estimates in our induction.

\begin{proposition}\label{prop:cycles-ic}
Let $X$ be a projective scheme over a field $k$, and let $\mathcal M$ be a coherent sheaf with $\dim\operatorname{Supp}\mathcal M\leq q$.

\begin{itemize}[leftmargin=1.5em]
\item[(1)] For an exact sequence
$0\to\mathcal M'\to\mathcal M\to\mathcal Q\to0$, one has $[\mathcal M]_q=[\mathcal M']_q+[\mathcal Q]_q$ as cycles on $X$.

\item[(2)]
If $L_1,\ldots,L_q$ are ample line bundles, then
{ \[
\deg\bigl( c_1(L_1)\cdots c_1(L_q)\cap[\mathcal M]_q \bigr)
= \sum_{\dim\overline{\{\eta\}}=q}  \operatorname{Length}_{{\mathcal O}_{X,\eta}}(\mathcal M_\eta) \deg\bigl( c_1(L_1)\cdots c_1(L_q) \cap[\overline{\{\eta\}}] \bigr).
\]}

\noindent If $\dim\operatorname{Supp}\mathcal M = 0$, then
\[
 \deg[\mathcal M]_0
 = \sum_{x\in\operatorname{Supp}\mathcal M} \operatorname{Length}_{{\mathcal O}_{X,x}}(\mathcal M_x)[\kappa(x):k]
 = \dim_k H^0(X,\mathcal M).
\]
In particular, over an algebraically closed field, this degree is the total length of $\mathcal M$.

\item[(3)]
If $A$ is an invertible sheaf, then $[\mathcal M\otimes A]_q=[\mathcal M]_q$.

\item[(4)]
Suppose $q\geq1$, and let $L$ be an invertible sheaf with a section $s\in H^0(X,L)$ such that multiplication by $s$ induces an injection $\mathcal M\otimes L^{-1}\xrightarrow{s}\mathcal M$. 
Denote $H=V(s)$ and regard $\mathcal M|_H$ as a sheaf on $X$.
Then
\[
 \dim\operatorname{Supp}(\mathcal M|_H)\leq q-1, \quad
 [\mathcal M|_H]_{q-1}  = c_1(L)\cap[\mathcal M]_q
\]
in $\operatorname{CH}_{q-1}(X)$.
Consequently, for line bundles $L_1,\ldots,L_{q-1}$,
{ \[
\deg\bigl( c_1(L_1)\cdots c_1(L_{q-1}) \cap[\mathcal M|_H]_{q-1} \bigr)
=
 \deg\bigl( c_1(L_1)\cdots c_1(L_{q-1})c_1(L) \cap[\mathcal M]_q \bigr).
\]}
\end{itemize}
\end{proposition}

\begin{proof}Assertions \textup{(1)} and \textup{(3)} follow from the definition of associated cycles and additivity of length~\cite[Tags 02QX and 02QZ]{Stacks}. 
Assertion \textup{(2)} follows from linearity, the definition of degree, and positivity of ample intersections \cite[Tag 0BEV]{Stacks}. Finally, \textup{(4)} follows from the local length formula for a regular section and its interpretation as intersection with the first Chern class~\cite[Tags 02QG and 02SN]{Stacks}. The intersection degree formula followed by intersection with the remaining Chern classes.
\end{proof}

A coherent sheaf is pure of dimension $q$ (cf. \cite[Definition 1.1.2]{HL10} and \cite[p. 49]{Matsumura80}) if every nonzero coherent subsheaf has support dimension $q$,
equivalently, all its associated points are generic points of $q$-dimensional support components. 
The following lemma gives multiplicity bounds for annihilators.

\begin{lemma} \label{lem:length}
Let ${\mathcal J}\subseteq{\mathcal O}_X$ be a coherent ideal sheaf, and let
$L_1,\ldots,L_q$ be ample line bundles.
If the coherent sheaf $\mathcal M$ is zero or pure of dimension $q$, and if $\operatorname{Supp}\mathcal M\subseteq |V({\mathcal J})|$, then
\[
 (\sqrt{{\mathcal J}})^{\deg\left( c_1(L_1)\cdots c_1(L_q)\cap[\mathcal M]_q \right)} \mathcal M=0.
\]
On the other hand, for any coherent sheaf $\mathcal G$ of dimension zero, one has
\[
 (\sqrt{\mathcal J})^{\deg[\mathcal G]_0}\mathcal G \subseteq\mathcal J\mathcal G.
\]
\end{lemma}

\begin{proof} For pure sheaf $\mathcal M$, by primary decomposition, there is an injection
\[
\mathcal M \hookrightarrow \oplus_{\eta\in \operatorname{Supp}(\mathcal M),\ \dim\overline{\{\eta\}}=q}\ \ \mathcal M_\eta.
\]
At each generic point $\eta$ of the support $\operatorname{Supp}\mathcal M$, the stalk $\mathcal M_\eta$ has finite length. Writing $\mathfrak m_\eta$ for the maximal ideal of $\mathcal O_{X,\eta}$, the finite-length argument gives
\[
\mathfrak m_\eta^{\operatorname{Length}(\mathcal M_\eta)} \mathcal M_\eta = 0.
\]
Since $(\sqrt{\mathcal J})_\eta \subseteq \mathfrak m_\eta$, it follows that
$(\sqrt{\mathcal J})_\eta^{\operatorname{Length}(\mathcal M_\eta)} \mathcal M_\eta = 0$.
Combining these observations and Proposition~\ref{prop:cycles-ic}, for ample line bundles $L_1,\ldots,L_q$,
\[
\deg c_1(L_1)\cdots c_1(L_q)\cap[\mathcal M]_q \geq \operatorname{Length}(\mathcal M_\eta)
\]
for every associated point $\eta$. Then purity forces $(\sqrt{{\mathcal J}})^{\deg\left( c_1(L_1)\cdots c_1(L_q)\cap[\mathcal M]_q \right)} \mathcal M$ to vanish everywhere.

For $\mathcal G$ with $\dim \mathcal G=0$, we apply the same finite-length argument at each support point of $\mathcal G/{\mathcal J}\mathcal G$.
This quotient is supported on $V({\mathcal J})$, and Proposition~\ref{prop:cycles-ic} \textup{(1)--(2)} gives 
$ \operatorname{Length}_{{\mathcal O}_{X,x}} \bigl((\mathcal G/{\mathcal J}\mathcal G)_x\bigr)
\leq  \deg[\mathcal G/{\mathcal J}\mathcal G]_0
 \leq  \deg[\mathcal G]_0$. Thus $(\sqrt{{\mathcal J}})^{\deg[\mathcal G]_0}$ annihilates $\mathcal G/{\mathcal J}\mathcal G$, which proves the second assertion.
\end{proof}

\subsection{Purity and cycle preservation of leading Ext} 

To prove the first assertion $\sqrt{\mathcal J}^d\subseteq \mathcal J$ in Theorem~\ref{thm:unified-intersection}, it suffices to show $(\sqrt{\mathcal J})^d\cdot\mathcal O_B=0$, where $B=V(\mathcal J)$ is the base scheme.
Lemma~\ref{lem:length} shows that generic multiplicities control annihilators once the sheaf is pure. To use this observation for the intermediate quotients in our induction, we therefore seek an operation that produces a pure sheaf while preserving its top-dimensional cycle. 
Local Grothendieck duality suggests such an operation.

Let $X$ be a projective Cohen-Macaulay scheme of pure dimension $n$ over $k$. Then its dualizing complex is concentrated in degree $-n$ and of the form $\omega_X[n]$ with $\omega_X$ a maximal Cohen-Macaulay coherent sheaf on $X$. We call $\omega_X$ its dualizing sheaf. Moreover, the natural homothety morphism gives $ R\mathcal Hom_X(\omega_X,\omega_X)\cong \mathcal O_X$. See~\cite[Tags 0A7B, 0AWS, and 0FVU]{Stacks}.
Recall that, for a closed immersion $i:Z\hookrightarrow X$, the right adjoint $i^!$ produces the dualizing complex $i^!(\omega_X[n])$ of $Z$. Its pushforward is computed by $i_*i^!(\omega_X[n])\cong R\mathcal Hom_X(i_*\mathcal O_Z,\omega_X[n])$. This brings Ext sheaves into the picture as the cohomology sheaves of a dualizing complex. 
For a general coherent sheaf $\mathcal M$ on $X$, we use the analogous derived dual $R\mathcal Hom_X(\mathcal M,\omega_X[n])$.
If $\dim\operatorname{Supp}\mathcal M=q$, local duality gives vanishing below degree $-q$, so truncation isolates the leading Ext sheaf:
\[
\tau^{\leq-q} R\mathcal Hom_X(\mathcal M,\omega_X[n])
\cong \mathcal{E}\!xt_X^{n-q}(\mathcal M,\omega_X)[q].
\]

The significance of this leading term is that it is pure of dimension $q$, even when $\mathcal M$ is not Cohen-Macaulay, and preserves the associated $q$-cycle. 
We therefore retain $\omega_X$ as the coefficient sheaf. Tensoring the resulting Ext sheaf with a line bundle preserves both properties. The following lemma records the statements we need.

\begin{lemma}\label{lem:leading}
Let $X$ be a projective Cohen-Macaulay scheme of pure dimension $n$ over a field $k$ with dualizing sheaf $\omega_X$, and let $0\leq q\leq n$.
For a coherent sheaf $\mathcal M$ with $\dim \operatorname{Supp} \mathcal M\leq q$, the following hold:
\begin{enumerate}[label=\textup{(\arabic*)}]
\item
${\mathcal{E}\!xt}_X^j(\mathcal M,{\omega}_X)=0$ for $j<n-q$;

\item
${\mathcal{E}\!xt}_X^{n-q}(\mathcal M,{\omega}_X)$ is zero or pure of dimension $q$;

\item
the equality of cycles $[{\mathcal{E}\!xt}_X^{n-q}(\mathcal M,{\omega}_X)]_q = [\mathcal M]_q$ holds on \(X\).
\end{enumerate}
If $\dim \operatorname{Supp} \mathcal M<q$, then
${\mathcal{E}\!xt}_X^{n-q}(\mathcal M,{\omega}_X)=0$.

\end{lemma}

\begin{proof}
The vanishing and purity assertions follow from the local canonical-module properties in \cite[Theorem~1.8 and Lemma~1.9(c)]{Schenzel}.
The same vanishing gives the assertion when $\dim \operatorname{Supp} \mathcal M$ is less than $q$. 
At the generic point $\eta$ of a $q$-dimensional component of $\operatorname{Supp} \mathcal M$, the ring ${\mathcal O}_{X,\eta}$ is Cohen-Macaulay of dimension $n-q$, and $\mathcal M_\eta$ has finite length.
Finite-length local duality~\cite[Tag 0A7Q]{Stacks} therefore gives
\[
 \operatorname{Length}_{{\mathcal O}_{X,\eta}} \bigl({\mathcal{E}\!xt}_X^{n-q}(\mathcal M,{\omega}_X)_\eta\bigr) = \operatorname{Length}_{{\mathcal O}_{X,\eta}}(\mathcal M_\eta).
\]
Since the Ext sheaf is supported on ${\operatorname{Supp}} \mathcal M$,
these equalities give \textup{(3)} by Definition~\ref{def:hilbert-coefficients}.
\end{proof}

\subsection{Evaluation image ideal} 
The evaluation image allows us to make comparisons between successive divisor cuts on the fixed base scheme $B=V(\mathcal J)$. The sheaves supported on $|B|$ will be controlled by Lemma~\ref{lem:length}, while the prescribed dimension conditions in Theorem~\ref{thm:unified-intersection} ensure that the next section is regular on the corresponding $|B|$-torsion-free quotient. Dual Koszul evaluations convert the resulting maps of Ext sheaves into maps of ideals in $\mathcal O_B$, through which annihilation statements pass from one stage to the next.

Throughout this subsection, let $X$ be a projective Cohen-Macaulay scheme of pure dimension $n$ over a field $k$, and let $\mathcal J\subseteq\mathcal O_X$ be a coherent ideal sheaf, with $B=V(\mathcal J)$. Let $L_1,\ldots,L_r$ be line bundles on $X$, with $0\leq r\leq n$, and choose $ s_i\in H^0(X,\mathcal{J}\otimes L_i)$. 
For $0\leq j\leq r$, write $B_j=V(s_1,\ldots,s_j)=V(\underline {s})$ for their common zero scheme, $B_0=X$. Thus $B\subseteq B_j$.
Each section $s_i$ gives a two-term complex $\left[ L_i^{-1}\xrightarrow{s_i}\mathcal{O}_X \right] $ concentrated in degrees $-1$ and $0$.
Their total tensor product is the Koszul complex 
$\mathsf{K}(\underline {s})= \mathsf{K}(s_1,\ldots,s_j) = \operatorname{Tot}\left( \otimes_{i=1}^{j} \left[ L_i^{-1}\xrightarrow{s_i}\mathcal{O}_X \right]\right)$, which takes the explicit form
\[
0 \to \bigotimes_{i=1}^{j}  L_i^{-1} \to \bigoplus_{\ell=1}^j \bigotimes_{\substack{i\ne\ell\\1\le i\le j}}  L_i^{-1} \to 
\cdots \to \bigoplus_{i=1}^j L_i^{-1} \xrightarrow{(s_1,\dots,s_j)} \mathcal O_X \to 0
\]
The complex is  concentrated in degrees $-j,\ldots,0$ and self dual under \[
R\mathcal Hom_{X}(\_,\mathcal O_X)\otimes_{i=1}^{j}L_i^{-1}[j],
\]see~\cite[Tag 062J]{Stacks}. 
When the sections form a regular sequence, this complex resolves $\mathcal{O}_{B_j}$. In general, one always has ${H}^{0}\bigl(\mathsf{K}(s_1,\ldots,s_j)\bigr) = \mathcal{O}_{B_j}$.

Let $\omega_X\twoheadrightarrow\mathcal F$ be a coherent quotient annihilated by $s_1,\ldots,s_j$, with $\dim\operatorname{Supp}\mathcal F\leq n-j$. 
Tensoring the Koszul augmentation map $\mathsf{K}(\underline {s}) \rightarrow \mathcal O_{B_j}$ with $\omega_X$ and composing with the induced quotient gives a chain map $\mathsf K(\underline {s})\otimes\omega_X \rightarrow\omega_X\otimes\mathcal O_{B_j} \twoheadrightarrow\mathcal F$. 
We call this the augmentation of the quotient $\omega_X\twoheadrightarrow\mathcal F$. Because the Koszul complex is locally free and self dual, the derived tensor-Hom adjunction gives 
$$R\mathcal Hom_X  \bigl(\mathsf K(\underline {s})\otimes\omega_X,\omega_X\bigr)\otimes_{i=1}^jL_i^{-1}[j]\cong  R\mathcal Hom_X  \bigl(\mathsf K(\underline {s}),\otimes_{i=1}^jL_i^{-1}[j]\bigr)\cong \mathsf K(\underline {s}).$$ 
Dualizing the augmentation by $R\mathcal Hom_{X}(\_,\omega_X\otimes_{i=1}^{j}L_i^{-1}[j])$, and taking cohomology in degree $0$ therefore gives a natural evaluation map 
\[ \mathrm{ev}_j\colon  \mathcal{E}\!xt_X^j(\mathcal F,\omega_X\otimes_{i=1}^jL_i^{-1}) \rightarrow\mathcal O_{B_j}.\] 
Its image is a coherent ideal sheaf of $\mathcal O_{B_j}$. Taking its image under $\mathcal O_{B_j}\twoheadrightarrow\mathcal O_B$ leads to the following definition.

\begin{definition}\label{def:image}
The $j$-th evaluation image ideal of the quotient $\omega_X\twoheadrightarrow\mathcal F$ in
$\mathcal{O}_B$ is
\[
 \mathfrak{q}_j(\mathcal F) := \operatorname{im}\left( \mathcal{E}\!xt_X^j(\mathcal F,\omega_X) \otimes_{i=1}^{j}L_i^{-1}
 \xrightarrow{\mathrm{ev}_j}  \mathcal{O}_{B_j} \twoheadrightarrow \mathcal{O}_B \right).
\]
For $j=0$, the empty Koszul complex and the empty tensor product are both $\mathcal{O}_X$. In particular, $ \mathfrak{q}_0(\omega_X)=\mathcal{O}_B$. 
\end{definition}

The definition does not require the sections $s_1,\ldots,s_j$ to form a regular sequence.

\begin{remark}
Locally, assume in addition that $X$ is Gorenstein, the evaluation ideal $\mathfrak q_j(\omega_X|_Z)$ for closed subscheme $Z\subseteq B_j$ is the image in $\mathcal O_B$ of the Koszul-Fitting ideal $\operatorname{Kitt}_R(\underline {s},I_Z)$ after trivializing the line bundles, see \cite{BH}.
\end{remark}

We now collect the basic properties of these evaluation image ideals.
\begin{lemma}\label{lem:basic-images}
Under the preceding hypotheses, let $\omega_X\twoheadrightarrow\mathcal F$ be a quotient with
\[
\dim \operatorname{Supp} \mathcal F\leq n-j.
\]
\begin{enumerate}[label=\textup{(\arabic*)}]
\item
The evaluation image ideal $\mathfrak q_j(\mathcal F)$ is coherent and supported on $\operatorname{Supp}(\mathcal F)\cap|B|$. Its construction commutes
with restriction to open subsets of $X$.

\item For a further quotient $\omega_X\to \mathcal F\twoheadrightarrow\mathcal F'$, one has $\mathfrak q_j(\mathcal F')\subseteq\mathfrak q_j(\mathcal F)$.
\end{enumerate}
\end{lemma}

\begin{proof}
Coherence follows since it is defined as the image of a morphism of coherent sheaves. The annihilator of $\mathcal F$ annihilates $\mathcal{E}\!xt_X^j(\mathcal F,\omega_X)$, and therefore annihilates its evaluation image as well. 
Since this image lies in $\mathcal O_B$, its support is contained in $\operatorname{Supp}\mathcal F\cap|B|$.
The Koszul construction, derived duality, and the operation of taking images all commute with restriction to open subsets. This proves \textup{(1)}.

For \textup{(2)}, the augmentation to $\mathcal F'$ factors as $\mathsf K(\underline {s})\otimes \omega_X\rightarrow\mathcal F\twoheadrightarrow\mathcal F'$. 
Contravariance of derived Hom therefore gives a morphism from the source of the evaluation for $\mathcal F'$ to that for $\mathcal F$, compatible with their evaluation maps into $\mathcal O_{B_j}$.
Composing with $\mathcal O_{B_j}\twoheadrightarrow\mathcal O_B$ and taking images yields the desired inclusion.
\end{proof}

We next compare evaluation image ideals before and after restriction to the zero scheme of the next section. When this section is regular  on the quotient sheaf being restricted, the evaluation images satisfy an inclusion in the fixed sheaf $\mathcal O_B$. This inclusion will allow us to pass from one stage of the induction to the next.

\begin{lemma}
\label{lem:cut}
Let $0\leq j<r$, and let $\omega_X\twoheadrightarrow\mathcal F$ be a quotient as above. Suppose that $s_{j+1}$ is regular on $\mathcal F$, and  $H=V(s_{j+1})$.
Regard $\mathcal F|_H$ as a sheaf on $X$, with its induced quotient map from $\omega_X$.
Then $s_1,\ldots,s_{j+1}$ annihilate $\mathcal F|_H$, $\dim\operatorname{Supp}(\mathcal F|_H)\leq n-j-1$, and
\begin{equation}\label{eq:cut-image}  \mathfrak q_j(\mathcal F)  \subseteq\mathfrak q_{j+1}(\mathcal F|_H).
\end{equation}
\end{lemma}

\begin{proof} Since the evaluation maps are obtained by applying derived duality to the Koszul augmentations, comparing their images requires us to understand how these augmentations interact with multiplication by $s_{j+1}$ and the resulting cut. Write $\underline s=(s_1,\ldots,s_j)$, and consider the
augmentation $\mathrm{aug}\colon \mathsf K(\underline s)\otimes \omega_X \rightarrow \mathcal O_{B_j}\otimes\omega_X \twoheadrightarrow \mathcal F$. On the Koszul side, adjoining $s_{j+1}$ amounts to tensoring with the two-term complex $\mathsf K(s_{j+1})=[L_{j+1}^{-1}\xrightarrow{s_{j+1}}\mathcal O_X]$, or equivalently, taking the mapping cone of multiplication by $s_{j+1}$:
\[
\mathsf K(\underline s,s_{j+1}) = \operatorname{Tot}\bigl( \mathsf K(\underline s)\otimes\mathsf K(s_{j+1})\bigr)
\cong \operatorname{Cone}\bigl(\mathsf K(\underline s)\otimes L_{j+1}^{-1}\xrightarrow{s_{j+1}} \mathsf K(\underline s)\bigr).
\]
On $\mathcal F$, the same construction computes restriction to $H$:
regularity of $s_{j+1}$ on $\mathcal F$ gives the exact
sequence $0\to \mathcal F\otimes L_{j+1}^{-1} \xrightarrow{s_{j+1}} \mathcal F \rightarrow \mathcal F|_H\to 0$ and also implies $\dim \mathcal F|_H\leq n-j-1$. The augmentation and quotient maps commute with multiplication by $s_{j+1}$ at the level of complexes. The induced maps of mapping cones therefore fit into the following commutative diagram in the derived category:
\[
\begin{tikzcd}[row sep={0.3cm}]
 {\mathsf K(\underline s)\otimes\omega_X\otimes L_{j+1}^{-1}}  \arrow[r,"s_{j+1}"]\arrow[d] & {\mathsf K(\underline s)\otimes\omega_X} \arrow[r]\arrow[d] & {\mathsf K(\underline s,s_{j+1})\otimes\omega_X} \arrow[d]\arrow[rd,"q_{j+1}"] & \\
 {\omega_X\otimes\mathcal O_{B_j}\otimes L_{j+1}^{-1}} \arrow[r,"s_{j+1}"]\arrow[d] & {\omega_X\otimes\mathcal O_{B_j}} \arrow[r]\arrow[d] & {\operatorname{Cone}(s_{j+1})}
 \arrow[d]\arrow[r,"q"]  & {\omega_X\otimes\mathcal O_{B_{j+1}}}\arrow[ld]\\ 
 {\mathcal F\otimes L_{j+1}^{-1}}\arrow[r,"s_{j+1}"]& {\mathcal F}\arrow[r] &
 {\mathcal F|_H} &
\end{tikzcd}
\]
Here the third column consists of morphisms in the derived category induced by the mapping cone construction, $q_{j+1}$ is the quotient map from the Koszul complex, and $q$ is the quotient map given by the cone construction. Thus the augmentation maps induce a commutative diagram of distinguished triangles:
\[
\begin{tikzcd}[row sep={0.5cm}]
	{\mathsf{K}(\underline{s})\otimes\omega_X\otimes L^{-1}_{j+1}} \arrow[r, "{s_{j+1}}"] \arrow[d, "{\mathrm {aug}\otimes \mathrm {id}}"] & {\mathsf{K}(\underline{s})} \otimes\omega_X \arrow[r] \arrow[d, "{\mathrm {aug}}"] & {{\mathsf{K}(\underline{s},s_{j+1})\otimes \omega_X}} \arrow[d, "{\mathrm {aug}}"] \arrow[r, "+1"] & {} \\
		{\mathcal F\otimes L^{-1}_{j+1}} \arrow[r, "{s_{j+1}}"]                                & {\mathcal F} \arrow[r]                          & {\mathcal F|_{H}} \arrow[r, "+1"]                                             & {}
	\end{tikzcd}.
	\]
To obtain the evaluation image ideals, we apply the derived dual
\[
R\mathcal Hom_X(-,\omega_X\otimes_{i=1}^{j+1}L_i^{-1}).
\]
The lower triangle together with the vanishing Lemma~\ref{lem:leading} (1)  gives the following long exact Ext sequence 
\[
\begin{aligned}
0 \longrightarrow{}&\mathcal{E}\!xt_X^{j}  (\mathcal F,\omega_X\otimes_{i=1}^{j}L_i^{-1})  \otimes L_{j+1}^{-1}\xrightarrow{\;s_{j+1}\;}\mathcal{E}\!xt_X^{j}  (\mathcal F,\omega_X\otimes_{i=1}^{j}L_i^{-1})\\[0.8ex]
\longrightarrow{}&\mathcal{E}\!xt_X^{j+1}  (\mathcal F|_H,\omega_X\otimes_{i=1}^{j+1}L_i^{-1})\longrightarrow\mathcal{E}\!xt_X^{j+1}  (\mathcal F,\omega_X\otimes_{i=1}^{j}L_i^{-1})  \otimes L_{j+1}^{-1}\longrightarrow \cdots. 
\end{aligned}
\]

\noindent Naturality of the connecting morphisms in the dualized diagram of triangles, together with Koszul self-duality, gives the following commutative diagram
\[
	\begin{tikzcd}[row sep=small]
		\dfrac{\mathcal{E}\!xt_X^j(\mathcal F,\omega_X\otimes_{i=1}^{j}L_i^{-1})}
		{s_{j+1}\mathcal{E}\!xt_X^j(\mathcal F,\omega_X\otimes_{i=1}^{j+1}L_i^{-1})}
		\arrow[r,hook]\arrow[d,"{\mathrm {ev}_j/(s_{j+1})}"]
		&
		\mathcal{E}\!xt_X^{j+1}(\mathcal F|_{H}, \omega_X\otimes_{i=1}^{j+1}L_i^{-1})
		\arrow[d,"{\mathrm {ev}_{j+1}}"]\\
		\mathcal O_{B_{j+1}}\arrow[r,equal]&\mathcal O_{B_{j+1}}.
	\end{tikzcd}
\]
Composing with the quotient map $\mathcal O_{B_{j+1}} \to \mathcal O_B$ and taking images, and and using that $s_{j+1}$ vanishes on $\mathcal O_B$, we conclude that
$
\mathfrak q_j(\mathcal F)\subseteq \mathfrak q_{j+1}(\mathcal F|_{H}).
$\end{proof}

\vspace{20pt}

\section{Proof of Theorem~\ref{thm:unified-intersection}}
\label{sec:all-dimensions}

We prove Theorem~\ref{thm:unified-intersection} by successively removing $|B|$-torsion and taking regular cuts of quotients of $\omega_X$.

\begin{proof}[Proof of Theorem~\ref{thm:unified-intersection}]
Let $U=X\setminus|B|$. Starting with $\mathcal F_0=\omega_X$, we construct a sequence of coherent quotient sheaves $\omega_X \twoheadrightarrow \mathcal F_1 \twoheadrightarrow\cdots \twoheadrightarrow \mathcal F_r$
by removing the subsheaves of sections supported on $|B|$ and then restricting to the zero scheme of the next prescribed section.

At the $j$-th stage with $0\leq j\leq r$,  we remove the $|B|$-torsion $\Gamma_{|B|}\mathcal F_j$ and define
\[
\mathcal F_j^{B\mathrm{tf}}=\mathcal F_j/\Gamma_{|B|}\mathcal F_j.
\]
Here the superscript ${B\mathrm{tf}}$ denotes $|B|$-torsion-free quotient, obtained by quotienting $\mathcal F_j$ by its $\mathcal J$-torsion subsheaf $\Gamma_{|B|}\mathcal F_j$. The subsheaf $\Gamma_{|B|}\mathcal F_j$ is coherent and supported on $B$, and $ \operatorname{Ass}\bigl(\mathcal F_j^{B\mathrm{tf}}\bigr)  \cap|B|=\varnothing$.  Suppose inductively that $\operatorname{Supp}\mathcal F_j\subseteq B_j$, $\mathcal F_j|_U=\omega_X\otimes \mathcal O_{B_j}|_{U}$, and $\dim \operatorname{Supp} \mathcal F_j\leq n-j$. 

For $j<r$, we verify that the next section $s_{j+1}$ is regular on $\mathcal F_j^{B\mathrm{tf}}$. Inductively, each  $B_j|_U$ is Cohen-Macaulay and pure of dimension $n-j$. Indeed, on $U$, the restriction of $s_{j+1}$ is regular on $\mathcal F_j^{B\mathrm{tf}}|_U$.  If $s_{j+1}$ vanished at an associated point of $B_j|_U$, then $B_{j+1}|_U$ would contain a component of dimension $n-j$, contrary to the prescribed dimension bound. Thus $s_{j+1}$ is regular on $B_j|_U$, and its  zero scheme is Cohen-Macaulay and pure of dimension $n-j-1$. The quotient map $\mathcal F_j\to\mathcal F_j^{B\mathrm{tf}}$ restricts to an isomorphism on $U$. 
All associated points of $\mathcal F_j^{B\mathrm{tf}}$ lie in $U$, where this sheaf agrees with $\omega_X\otimes\mathcal O_{B_j}|_U$. Hence $s_{j+1}$ is regular on $\mathcal F_j^{B\mathrm{tf}}$.

For $0\leq j<r$, we define $$\mathcal F_{j+1} = \mathcal F_j^{{B \mathrm {tf}}}/  s_{j+1}\mathcal F_j^{{B \mathrm {tf}}}\otimes L_{j+1}^{-1}.$$
Then $\operatorname{Supp}\mathcal F_{j+1}\subseteq B_{j+1}$, $\mathcal F_{j+1}|_U=\omega_X\otimes \mathcal O_{B_{j+1}}|_U$ and $\dim \operatorname{Supp} \mathcal F_{j+1}\leq n-j-1$. If $r<n$, condition~\textup{(2)} of Theorem~\ref{thm:unified-intersection} also gives $\operatorname{Supp}\mathcal F_r\subseteq|B|$.

Consider now the exact sequence $0\to \Gamma_{|B|}\mathcal F_j\to\mathcal F_j\xrightarrow{ }  \mathcal{F}_j^{B\mathrm {tf}}\to 0$.  
Taking local duality $\mathcal E(\cdot):=\mathcal{E}\!xt^j_X(\cdot,\omega_X\otimes_{i=1}^{j}L_i^{-1})$, by Lemma~\ref{lem:leading}\textup{(1)}, we have
{\small
\[\begin{tikzcd}
	0 \arrow[r] & \mathcal E {\mathcal F_j^{{B \mathrm {tf}}}} \arrow[r] &  \mathcal {E}\mathcal{F}_j \arrow[rr] \arrow[rd, two heads] &                                    & {\mathcal E{\Gamma_{|B|}\mathcal F_j}} \arrow[r] & {\cdots} \\
	&                                                                  &                                                                         & {{\mathcal Q_j}} \arrow[ru, hook] &                                                            &         
\end{tikzcd}.\]}

\noindent Here we adopt the shorthand notation ${\mathcal Q_j}$ to denote the cokernel of the dual inclusion map $\mathcal E  \mathcal F_j^{{B \mathrm {tf}}} \hookrightarrow \mathcal E \mathcal F_j$.  By Lemma~\ref{lem:leading} \textup{(2)} the leading dual modules $\mathcal E{\mathcal F_j^{{B \mathrm {tf}}}}$, $\mathcal{E}\mathcal{F}_j$ and ${\mathcal E{\Gamma_{|B|}\mathcal F_j}}$ are zero or of pure dimension $q=n-j$. The sheaf  ${\mathcal E{\Gamma_{|B|}\mathcal F_j}}$ is supported on $|B|$, thus the cokernel ${{\mathcal Q_j}}$ is also zero or of pure dimension, and is supported on $|B|$. Applying Lemma~\ref{lem:length} to ${\mathcal Q_j}$,
with the ample line bundles
$L_{j+1},\ldots,L_r$ and $n-r$ copies of $A$, yields
\begin{equation} \label{eq:cokervanish}
     (\sqrt{\mathcal{J}})^{\deg\left(\prod_{i=j+1}^{r}c_1(L_i)
 c_1(A)^{n-r}\cap [{\mathcal Q_j}]_q\right)}
 {\mathcal Q_j}=0.
\end{equation}
Cycle preservation under duality (cf. Lemma~\ref{lem:leading} \textup{(3)}) and invariance under invertible twists (Proposition~\ref{prop:cycles-ic} \textup{(1), (3)}),
together with additivity, give  $[{\mathcal Q_j}]_{n-j} =  [\mathcal F_j]_{n-j}-[\mathcal F_j^{{B \mathrm {tf}}}]_{n-j}$.
Define
\[
 e_j =\deg\left( \left(\prod_{i=j+1}^{r}c_1(L_i)\right) c_1(A)^{ n-r}\cap[\mathcal F_j]_{n-j} \right),\ 0\leq j\leq r.
\]
These are nonnegative integers by
Proposition~\ref{prop:cycles-ic} \textup{(2)}, and $e_0=d$.
For $j<r$, Proposition~\ref{prop:cycles-ic} \textup{(4)} gives 
$ [\mathcal F_{j+1}]_{n-j-1}
 = c_1(L_{j+1})\cap[\mathcal F_j^{{B \mathrm {tf}}}]_{n-j}$ in $\operatorname{CH}_{n-j-1}(X)$. Thus
 \begin{align*}
     e_{j+1}\!=&\deg\left( (\prod_{i=j+2}^{r}c_1(L_i)) c_1(A)^{ n-r}\cap[\mathcal F_{j+1}]_{n-j-1} \right)\\
 =& \deg\left( (\prod_{i=j+1}^{r}c_1(L_i)) c_1(A)^{n-r}\cap[\mathcal F_j^{{B \mathrm {tf}}}]_{n-j} \right).
\end{align*}
Consequently,
\[
 \deg\left(  \left(\prod_{i=j+1}^{r}c_1(L_i)\right) c_1(A)^{n-r}\cap[{\mathcal Q_j}]_{n-j} \right)= e_j-e_{j+1}.
\]
So equation (\ref{eq:cokervanish}) yields $(\sqrt{\mathcal J})^{e_j-e_{j+1}}{\mathcal Q_j}=0$. Since this power annihilates the cokernel $\mathcal Q_j$, so $(\sqrt{\mathcal J})^{e_j-e_{j+1}}(\mathcal{E} \mathcal{F}_j)\subseteq (\mathcal{E}\mathcal {F}_j^{{B \mathrm {tf}}})$. Passing to the evaluation images (Definition~\ref{def:image}), we have
 $(\sqrt{\mathcal{J}})^{e_j-e_{j+1}} \mathfrak{q}_j(\mathcal F_j)
 \subseteq \mathfrak{q}_j(\mathcal F_j^{{B \mathrm {tf}}})$. 
 Multiplying by $(\sqrt{\mathcal{J}})^{e_{j+1}}$, we obtain
 \[
 (\sqrt{\mathcal{J}})^{e_j}\mathfrak{q}_j(\mathcal F_j)
 \subseteq
 (\sqrt{\mathcal{J}})^{e_{j+1}} \mathfrak{q}_j(\mathcal F_j^{{B \mathrm {tf}}}).
 \]
 Combining the equation~\eqref{eq:cut-image}, we obtain the reduction
\[
 (\sqrt{\mathcal{J}})^{e_j}\mathfrak{q}_j(\mathcal F_j)
 \subseteq
 (\sqrt{\mathcal{J}})^{e_{j+1}}
 \mathfrak{q}_{j+1}(\mathcal F_{j+1}), \ 0\leq j<r.
\]
Iteration, including $r=0$, yields
\begin{equation}\label{eq:unified-iteration}
 (\sqrt{\mathcal J})^{e_0}\mathcal O_B  \subseteq(\sqrt{\mathcal J})^{e_r}\mathfrak q_r(\mathcal F_r).
\end{equation}
If $r<n$, the sheaf $\mathcal E(\mathcal F_r)$ is pure or zero, is supported on $|B|$, and has weighted degree $e_r$ with respect to $n-r$ copies of $A$.  By Lemma~\ref{lem:length}, the right-hand side of \eqref{eq:unified-iteration} equals zero.  
If $r=n$, the sheaf $\mathcal E(\mathcal F_n)$ has support of dimension at most zero, and $\deg [\mathcal E(\mathcal F_n)]_0=e_n$. The zero-dimensional assertion of Lemma~\ref{lem:length} therefore gives $(\sqrt{\mathcal J})^{e_n}\mathcal E(\mathcal F_n) \subseteq \mathcal J\mathcal E(\mathcal F_n)$, whose image under evaluation in $\mathcal O_B$ is zero. The right-hand side of \eqref{eq:unified-iteration} also vanishes.  
Finally, $e_0=d$, and $\mathcal O_B=\mathcal O_X/\mathcal J$, so $\sqrt{\mathcal J}^d\mathcal O_B=0$ is exactly $\sqrt{\mathcal J}^d\subseteq \mathcal J$.

To prove the prime-power product assertion, we retain the generic multiplicities in the comparison with the $|B|$-torsion-free quotient. Fix a stage $0\leq j\leq r$.
The cokernel $\mathcal Q_j$ at this stage is zero or pure of dimension $n-j$. Decompose its effective cycle as 
\begin{equation}\label{eq:ajcycledec}[\mathcal F_j]_{n-j}-[\mathcal F_j^{{B \mathrm {tf}}}]_{n-j}= [\mathcal Q_j]_{n-j}=\sum_{\alpha\in \mathfrak A_j}m_\alpha[W_\alpha],\end{equation}
where the $W_\alpha$ are the irreducible components of $\operatorname{Supp}\mathcal Q_j$ carries their reduced induced scheme structures and $m_\alpha>0$ is the generic length of the $\mathcal Q_j$ along $W_\alpha$. We index by $\mathfrak A=\cup_j \mathfrak A_j$ for the components of $\alpha$'s occurring over all stages. Moreover, one record the stage by the codimensions. That is $\mathfrak A_j=\{\alpha\mid \mathrm {codim}_X W_\alpha=j\}$. Each $W_\alpha$ lies in $|B|$. Set $\mathfrak a_j= \prod_{\alpha\in \mathfrak A_j}\mathcal I_{W_\alpha}^{m_\alpha}$, with an empty product understood to be $\mathcal O_X$.
The finite-length argument in the proof of Lemma~\ref{lem:length} shows that $\mathfrak a_j$ annihilates the cokernel $\mathcal Q_j$ at every generic point of its support. Purity therefore gives annihilation $\mathfrak a_j \mathcal Q_j=0$ everywhere.
Passing to evaluation images, we obtain
$\mathfrak a_j\mathfrak q_j(\mathcal F_j)
\subseteq
\mathfrak q_j(\mathcal F_j^{{B \mathrm {tf}}})$, $0\leq j\leq r$. 
Consequently, for $j<r$, Lemma~\ref{lem:cut} gives
\[    \mathfrak a_j\mathfrak q_j(\mathcal F_j)\subseteq\mathfrak q_j(\mathcal F_j^{{B \mathrm {tf}}})\subseteq \mathfrak q_{j+1}(\mathcal F_{j+1}).\]
At the final stage, if $r<n$, $\mathcal F_r^{{B \mathrm {tf}}}=0$; if $r=n$, it has dimension at most zero and is disjoint from $|B|$.
Thus Lemma~\ref{lem:basic-images} \textup{(1)} gives $\mathfrak q_r(\mathcal F_r^{{B \mathrm {tf}}})=0$, $\mathfrak a_r\mathfrak q_r(\mathcal F_r)=0$. Iteration from $\mathfrak q_0(\omega_X)=\mathcal O_B$ yields $\left(\prod_{j=0}^{r}\mathfrak a_j\right)\mathcal O_B=0$. Since $\mathcal O_B=\mathcal O_X/\mathcal J$, this is equivalent to $\prod_\alpha\mathcal I_{W_\alpha}^{m_\alpha} =\prod_{j=0}^{r}\mathfrak a_j \subseteq\mathcal J$.

For each $\alpha$, let $w_\alpha =
\deg \left( \left(\prod_{i=\mathrm {codim}_X W_\alpha+1}^{r}c_1(L_i)\right) c_1(A)^{n-r}\cap[W_\alpha] \right)$. The cycle decompositions (\ref{eq:ajcycledec}) gives $\sum_{\alpha\in \mathfrak A_j}m_\alpha w_\alpha = e_j-e_{j+1}$, $0\leq j<r$. At the final stage, 
\[
\sum_{\alpha\in \mathfrak A_r}m_\alpha w_\alpha 
=\deg \left( c_1(A)^{n-r}\cap \bigl([\mathcal F_r]_{n-r}-[\mathcal F_r^{{B \mathrm {tf}}}]_{n-r}\bigr)\right)\leq e_r,
\]
where the inequality follows from the nonnegativity of the degree contributed by $[\mathcal F_r^{{B \mathrm {tf}}}]_{n-r}$.
Summing over all stages, we conclude that 
\[
\sum_\alpha m_\alpha w_\alpha \leq \sum_{j=0}^{r-1}(e_j-e_{j+1})+e_r = e_0=d.
\]

Finally, every $W_\alpha$ lies in $|B|$, while the product containment gives the reverse inclusion after taking zero loci. Hence $\cup_\alpha|W_\alpha|=|B|$, which completes the proof.
\end{proof}

\vspace{20pt}

\section{Corollaries of Theorem~\ref{thm:unified-intersection}}
\label{sec:refined-degrees}

We now give some applications of Theorem~\ref{thm:unified-intersection} by choosing suitable line bundles and sections. This yields radical-power bounds for fewer generators and generators of unequal degrees.
When the choice of auxiliary sections requires an infinite field, we pass to $k(t)$; the relevant intersection degrees are unchanged, and the resulting radical-power containments descend to $k$ by faithful flatness. 

\begin{proof}[Proof of Corollary~\ref{thm:fewer-generators}]
After replacing $k$ by $k(t)$ if necessary, we may assume that $k$ is infinite. Let $f_1,\ldots,f_m$ generate $\mathcal J\otimes L$, and take $B=V(\mathcal J)$. We can choose $r=\min\{m,n\}$ linear combinations $$(s_1,\ldots,s_r)=(f_1,\ldots ,f_m)C$$ with $C\in M_{m\times r}(k)$ of full rank, such that each successive section is regular on the preceding common zero scheme outside $|B|$.
Indeed, at each associated point outside $|B|$, the generators $f_1,\ldots,f_m$ do not all vanish. Thus the vanishing of the next linear combination rules out a proper linear subspace of the coefficient space. The span of the preceding coefficient vectors is also proper. Avoiding these finitely many subspaces gives the required choice.

Since $X$ is Cohen-Macaulay and pure-dimensional, successive regular sections give $\dim\bigl(V(\underline s)\setminus|B|\bigr)\leq n-j$, $1\leq j\leq r$. If $r=m<n$, the coefficient matrix is invertible, so the new sections generate the same ideal sheaf and $V(s_1,\ldots,s_m)=B$.
Thus the condition~\textup{(2)} of Theorem~\ref{thm:unified-intersection} also holds.
Apply Theorem~\ref{thm:unified-intersection} with $L_1=\cdots=L_r = A = L$. The last assertion follows from $c_1(A^a)=a \cdot c_1(A)$.
\end{proof}

\begin{proof}[Proof of Corollary~\ref{thm:mixed-degrees}]
Let $r=\min\{m,n\}$ and $B=V(\mathcal J)$. Take $s_1=f_m$, which is regular because $X$ is integral. For $1\leq j<r$, choose successively
\begin{equation}\label{eq:unified-triangular-sections}
 s_{j+1} =f_j+\sum_{i=j+1}^{m-1}u_{ji}f_i,\quad  u_{ji} \in H^0(X,L^{a_j-a_i}).
\end{equation}
At any associated point of $V(s_1,\ldots,s_j)$ outside $|B|$, the sections $f_j,\ldots,f_{m-1}$ cannot all vanish:
otherwise the preceding triangular equations, read backwards, would force all the $f_i$ to vanish at that point. Since each $L^{a_j-a_i}$ is globally generated, the condition that $s_{j+1}$ vanish at such a point is either empty or a proper affine linear condition. Avoiding the finitely many proper affine subspaces makes $s_{j+1}$ regular on the preceding common zero scheme outside $|B|$. Hence all the dimension conditions of Theorem~\ref{thm:unified-intersection} hold.

If $r=m<n$, then after local trivialization of $L$, the full list $(s_1,\ldots,s_m)$ is obtained from $(f_m,f_1,\ldots,f_{m-1})$ by an invertible triangular change. Consequently, $V(s_1,\ldots,s_m)=B$, which gives the condition~\textup{(2)} of Theorem~\ref{thm:unified-intersection}. 
Take $L_1=L^{a_m}$, $L_{j+1}=L^{a_j}$, $1\leq j<r$, $A=L$. Apply Theorem~\ref{thm:unified-intersection} to these line bundles, we obtain the exponent
\[
\deg\left( c_1(L^{a_m}) \prod_{j=1}^{r-1}c_1(L^{a_j})c_1(L)^{n-r}\cap[X]\right) = \deg_L X\cdot a_m a_1\cdots a_{r-1},
\]
as required.
\end{proof}

\begin{remark}
For Corollary~\ref{thm:mixed-degrees}, it suffices that the difference powers
$L^{a_j-a_i}$ appearing in \eqref{eq:unified-triangular-sections} are globally generated. In particular, if $n=1$ or $m\le 2$, so $L$ only needs to be ample.
\end{remark}

The proof of Corollary~\ref{thm:mixed-degrees} uses global generation of the difference powers to choose its triangular sections. For ample line bundles, the following result gives bounds using a finite flat root cover and higher-degree auxiliary sections.

\begin{corollary}
\label{thm:ample-mixed}
Let $X$ be an integral projective Cohen-Macaulay scheme of pure dimension $n\geq1$ over a field $k$. Let $L$ be an ample line bundle, and let $f_i\in H^0(X,L^{a_i})$ be nonzero sections with $a_1\geq\cdots\geq a_m\geq1$ that generate $\mathcal J$. Then 
$$(\sqrt{\mathcal J})^{ \deg_L X \cdot a_1\cdots a_m}  \subseteq\mathcal J.$$ 
If $m>n$, choose $b\geq a_1$ such that $L^{b-a_i}$ is globally generated for every $i<m$. Then
$$(\sqrt{\mathcal J})^{\deg_L X \cdot a_m b^{n-1}} \subseteq\mathcal J.$$ 
Thus the smaller of these two exponents is valid. In particular, the exponent of Corollary~\ref{thm:mixed-degrees} remains valid for merely ample $L$ whenever $m\leq n$.
\end{corollary}

\begin{proof}
Consider the finite flat cover $\pi:Y\to X$ obtained by adjoining roots $\sigma_i^{a_i}=\pi^*f_i$. Its algebra is $ \pi_*\mathcal O_Y  =  \bigotimes_{i=1}^m \left(\bigoplus_{q=0}^{a_i-1}L^{-q}\right)$, with multiplication determined by the sections $f_i$. The defining equations are successively monic in new variables, so they form a regular sequence over the local polynomial rings of $X$. Thus $Y$ is projective, Cohen-Macaulay and pure of dimension $n$, and $\pi^*L$ is ample. Finite flatness and the projection formula give $\pi_*[Y] =  a_1\cdots a_m[X]$, $\deg_{\pi^*L}(Y) = a_1\cdots a_m\deg_L X$, see \cite[Tags 02RH and 02SU]{Stacks}.

Let $\mathcal I\subseteq\mathcal O_Y$ be the ideal sheaf generated by the root sections $\sigma_1,\ldots,\sigma_m$. Then $\mathcal I\otimes\pi^*L$ is globally generated. Applying Corollary~\ref{thm:fewer-generators} with both ample line bundles taken to be $\pi^*L$, we obtain $ (\sqrt{\mathcal I})^{\deg_{\pi^*L}(Y)}  \subseteq\mathcal I$. Since $\mathcal J\mathcal O_Y\subseteq\mathcal I$, it follows that $ (\sqrt{\mathcal J})^{ \deg_L Xa_1\cdots a_m}\mathcal O_Y \subseteq\mathcal I$.

The ideal $\mathcal I$ contracts to $\mathcal J$:
locally, quotienting the root algebra by all the $\sigma_i$ gives $R/(f_1,\ldots,f_m)$. This proves the first assertion.

For the second assertion, take $s_1=f_m$. It is regular because $X$ is integral. On $V(f_m)\setminus|B|$, the linear system $\sum_{i=1}^{m-1}
 f_i H^0(X,L^{b-a_i})  \subseteq H^0(X,\mathcal J\otimes L^b)$ has no base points. As in the proof of Corollary~\ref{thm:fewer-generators}, choose $n-1$ successive sections $s_2,\ldots,s_n$ that are regular on the preceding common zero scheme outside $|B|$.
Applying Theorem~\ref{thm:unified-intersection} with $r=n$, $L_1=L^{a_m}$ and $ L_2=\cdots=L_n=L^b$, we obtain $(\sqrt{\mathcal J})^{a_m b^{n-1}\deg_L X}\subseteq\mathcal J$. When \(n=1\), only \(s_1\) is needed. This proves the second assertion.
\end{proof}

We finish this section with examples showing that the exponents in Corollary~\ref{thm:fewer-generators}, Corollary~\ref{thm:degree}, and Corollary~\ref{thm:mixed-degrees} are sharp. The first example, which is adapted from Koll\'ar's Example~2.3 in \cite{Kollar88}, realizes the sharp exponents on smooth projective hypersurfaces.

\begin{example}
\label{ex:ambient-sharp}
For $1\leq m\leq n$ and $e\geq2$, set
\[
X=V(x_mx_0^{e-1}+\sum_{i=1}^m x_i^e-x_{n+1}^e)\subseteq\mathbb P_{\mathbb C}^{n+1}
\]
and $L=\mathcal O_X(1)$.
This hypersurface is integral and Cohen-Macaulay of dimension $n$, with $\deg_L X=e$. For $a_1\geq\cdots\geq a_m\geq1$, let $\mathcal J$ be generated by the sections $f_1=x_1^{a_1}$ and $f_i=x_{i-1}x_0^{a_i-1}-x_i^{a_i}$ for $2\leq i\leq m$. At $p=[1:0:\cdots:0]$, we work in the affine chart $x_0=1$. Eliminating $x_1,\ldots,x_{m-1}$ gives $x_m^{a_1\cdots a_m}=0$ and $(x_{n+1})^e=x_mu(x_m)$, where $u(0)=1$.
The formal inverse-function theorem therefore gives $\widehat{\mathcal O_{X,p}/\mathcal J_p} \cong\mathbb C[[x_{m+1},\ldots,x_{n+1}]]/ (x_{n+1}^{e a_1\cdots a_m})$. Its nilradical has nilpotence index $e a_1\cdots a_m$. By Corollary~\ref{thm:mixed-degrees}, the exponent $e a_1\cdots a_m$ is attained.

When $m=n$, taking all $a_i=1$ attains the exponent $\deg_L X$ in Corollary~\ref{thm:degree}.

For $m>n$, apply the same construction to the degrees $a_1,\ldots,a_{n-1},a_m$.
Denote the resulting generators by $f_1,\ldots,f_{n-1},f_m$, and set $f_i=x_0^{a_i-a_m}f_m$ for $n\leq i\leq m-1$. These additional sections have the prescribed degrees and do not change the ideal sheaf. Thus the exponent $e a_1\cdots a_{n-1}a_m$ is attained as well, for $n=1$, this exponent is $e a_m$.

For $e=1$ and $m\leq n$, consider the sections $f_1=x_1^{a_1}$ and $f_i=x_{i-1}x_0^{a_i-1}-x_i^{a_i}$ for $2\leq i\leq m$ on $\mathbb P^n$, adding multiples as above when $m>n$. Taking all $a_i=a$ gives the exponent $a^{\min\{m,n\}}$ in Corollary~\ref{thm:fewer-generators}.
\end{example}

The Cohen-Macaulay hypothesis is essential for the degree bound, as the following example shows.

\begin{example} \label{ex:cm-hypothesis}
Over an algebraically closed field $k$, set
\[
S=k[x+y,z,xy,x^2y,xz]\subseteq k[x,y,z]
\]
with the induced grading, and $X=\operatorname{Proj}S$. The finite birational morphism $\nu:\mathbb P^2\to X$ identifies $[1:0:0]$ and $[0:1:0]$ to a point $o$ and it is an isomorphism elsewhere. 
The scheme $X$ is integral but not Cohen-Macaulay. The sheaf $L=\mathcal O_X(1)$ is invertible: 
it is trivialized by $x+y$ near $o$, and invertibility elsewhere follows from $\nu$. Since $\nu^*L=\mathcal O_{\mathbb P^2}(1)$, $L$ is ample and $\deg_LX=1$. Let $\mathcal J$ be the ideal defined by $z\in H^0(X,L)$.
Then $\mathcal J\otimes L\cong\mathcal O_X$. On $D_+(x+y)$, let $t=x/(x+y)$ and $w=z/(x+y)$, the coordinate ring is $R=\{f\in k[t,w]:f(0,0)=f(1,0)\}$ and $\mathcal J$ corresponds to $(w)$. Here $tw\notin(w)$, since $t\notin R$, but $(tw)^2=w(t^2w)\in(w)$. Thus $\mathcal J$ is not radical, and exponent $\deg_LX=1$ does not suffice.
\end{example}

\vspace{20pt}

\section{Sharp bounds in graded Cohen-Macaulay rings}
\label{sec:graded-sharp-bounds}

We now pass from ideal sheaves to homogeneous ideals. We use the same evaluation images and perform one final interpolation on a finite scheme, thereby deriving the homogeneous and affine version.
For a finitely generated standard graded $k$-algebra $S=\bigoplus_{\nu=0}^{\infty}S_\nu$ with $S_0=k$ and  $S_+=\bigoplus_{\nu=1}^{\infty}S_\nu$, we write $I^{\mathrm{sat}}=I:S_+^\infty$ for a homogeneous ideal $I\subseteq S$. On the geometric side, let $Y=\operatorname {Proj}S$. For a closed subscheme $Z\subseteq Y$, let $I_Z\subseteq S$ denote its saturated homogeneous defining ideal.

\begin{proof}[Proof of Theorem~\ref{thm:graded-homogeneous-containment}]
As in Section~\ref{sec:refined-degrees}, we may assume that $k$ is infinite after passing to a purely transcendental extension of $k$. Choose the triangular sections
\[
s_{1}, \ldots, s_{\min\{m,n+1\}}
\]
as in \eqref{eq:unified-triangular-sections} with $s_1=f_m$ and $s_{j+1}\in I\cap S_{a_j}$ for  $j=1,\ldots ,\min\{m-1,n\}$. The same avoidance argument yields successive regular cuts outside $V_Y(I)$. If $m\leq n$, these sections generate $I$; if $m>n$, the last section $s_{n+1}$ does not vanish on $V_Y(s_1,\ldots,s_n)\setminus V_Y(I)$.

Consider the projective cone $X=\operatorname{Proj}S[t]$ with $\deg t=1$. Take $L=\mathcal O_X(1)$ and $\mathcal J=\widetilde{IS[t]}$ with $B=V(\mathcal J)$. 
Then $X$ is an integral projective Cohen-Macaulay scheme of dimension $n+1$ and degree $\deg Y$.
We identify $Y$ with $V_X(t)$ and $D_+(t)$ with $\operatorname{Spec}S$.
The chosen sections satisfy the conditions~\textup{(1)} and~\textup{(2)} of Theorem~\ref{thm:unified-intersection} on $X$: away from the vertex the projection to $Y$ has one-dimensional fibers, and the vertex belongs to $B$.

If $m\leq n$, applying Theorem~\ref{thm:unified-intersection} on $X$, and restricted to $D_+(t)$, we obtain $(\sqrt I)^{\deg Y \cdot a_1\cdots a_m}\subseteq I$. Since $\deg Y \cdot a_1\cdots a_m  \geq\deg Y+\sum_{i=1}^m(a_i-1)$, $\deg Y \cdot a_1\cdots a_m$ equals $d$.

Suppose henceforth that $m>n$.
Starting with $\mathcal F_0=\omega_X$, apply the $|B|$-torsion removal and regular cut construction of Section~\ref{sec:all-dimensions} for the first $n$ sections $s_1,\ldots,s_n$.
For geometric convenience, we construct closed subschemes $Z_i$, in parallel with the sheaf-theoretic construction of $\mathcal F_i$. Let $Z_0 = X$ and take $\mathcal O_{Z_j^{{B \mathrm {tf}}}} =\mathcal O_{Z_j}/\Gamma_{|B|}\mathcal O_{Z_j}$, $Z_{j+1}=Z_j^{{B \mathrm {tf}}}\cap V_X(s_{j+1})$. The sections are regular on both $\mathcal F_j^{B\mathrm{tf}}$ and $\mathcal O_{Z_j^{B\mathrm{tf}}}$. Inductively, $\mathcal I_{Z_j}\mathcal F_j=0$, $[\mathcal F_j]_{n+1-j}=[Z_j]_{n+1-j}$ and the cycle equality also holds after removing $|B|$-torsions.
Indeed, taking $|B|$-torsion free quotient preserves the annihilation relation and removes the same top-dimensional components supported on $|B|$.
Starting from $[\omega_X]_{n+1}=[X]$, equality of the generic multiplicities is then preserved by the local length formula for each regular cut \cite[Tag 02QG]{Stacks}.

The degree calculation and the inclusions between evaluation image ideals from Section~\ref{sec:all-dimensions} give
\begin{equation} \label{eq:graded-degree-budget}
\begin{aligned}
    &\deg Y\cdot a_m a_1\cdots a_{n-1}\\=&\deg_L[Z_n]_1+\textstyle\sum_{j=1}^{n-1} (a_j\cdots a_{n-1})\deg_L\left([Z_j]_{n+1-j}-[Z_j^{{B \mathrm {tf}}}]_{n+1-j}\right)
\end{aligned}
\end{equation}
and, on $D_+(t)=\operatorname{Spec}S$,
\begin{equation}
\label{eq:graded-image-containment}
(\sqrt I)^{\sum_{j=1}^{n-1}\deg_L([Z_j]_{n+1-j}-[Z_j^{{B \mathrm {tf}}}]_{n+1-j})}(S/I)
\subseteq
\mathfrak q_n(\mathcal F_n)|_{t\ne0}.
\end{equation}
Here we identify the evaluation image ideal $\mathfrak q_n(\mathcal F_n)|_{t\ne0}$ on this affine chart with the corresponding ideal of $S/I$.

Let $I_{Z_n}\subseteq S[t]$ be the saturated homogeneous defining ideal of $Z_n$, and $K=I_{Z_n}\cap S$.
The cuts and the formation of the $|B|$-torsion-free quotients commute with adjoining $t$, and every extended ideal is saturated with respect to $(S[t])_+$. Hence $I_{Z_n}=KS[t]$, $Z_n\cap Y=\operatorname{Proj}(S/K)$, and $I_{Z_n\cap Y}=K^{\mathrm{sat}}$, with $\dim S/K\leq1$.

Let $F_n=\Gamma(D_+(t),\mathcal F_n)$, and let $\omega_S$ be a canonical module of $S$.
On $D_+(t)$, after trivializing $\mathcal O_X(1)$ twists by powers of $t$, the evaluation image is obtained from $\operatorname{Ext}_S^n(F_n,\omega_S)\to S/I$. Since $KF_n=0$, its source is zero or pure of dimension one by Lemma~\ref{lem:leading}.

We then prove the containment $\sqrt{I}^d\subseteq I$. If $Z_n\cap Y=\varnothing$, then $\dim F_n\leq0$, so $\operatorname{Ext}_S^n(F_n,\omega_S)=0$ by Lemma~\ref{lem:leading} \textup{(1)}. Thus the right-hand side of \eqref{eq:graded-image-containment} is zero. Since \eqref{eq:graded-degree-budget} bounds the exponent on its left by $d$, the assertion follows.

So we assume henceforth that $Z_n\cap Y\ne\varnothing$. By \eqref{eq:graded-image-containment}, it suffices to prove 
\begin{equation}
\label{eq:graded-terminal-vanishing}
(\sqrt I)^{d-\sum_{j=1}^{n-1}\deg_L([Z_j]_{n+1-j}-[Z_j^{{B \mathrm {tf}}}]_{n+1-j})}\mathfrak q_n(\mathcal F_n)|_{t\ne0}=0.
\end{equation}
By \eqref{eq:graded-degree-budget} and the definition of $d$, the exponent here is at least $\deg_L[Z_n]_1$. Moreover, $K^\mathrm {sat}$ annihilates the evaluation image ideal $\mathfrak q_n(\mathcal F_n)|_{t\neq 0}$. The module $\operatorname{Ext}_S^n(F_n,\omega_S)$ is annihilated by $K$, and $K^{\mathrm{sat}}/K$ has finite length. The module $K^{\mathrm{sat}}\operatorname{Ext}_S^n(F_n,\omega_S)$ therefore has finite length and must vanish by purity. Consequently, $K^{\mathrm{sat}}\cdot\mathfrak q_n(\mathcal F_n)|_{t\neq0}$ equals to zero.

Consider the following cases.

\emph{Case 1: Every point of $Z_n\cap Y$ lies in $V_Y(I)$.}
On each local Artin component of $Z_n\cap Y$, the localization of $\sqrt I$ lies in the maximal ideal. By Lemma~\ref{lem:length}, $(\sqrt I)^{\deg_L[Z_n]_1} \subseteq I_{Z_n\cap Y}=K^{\mathrm{sat}}$. Since $K^{\mathrm{sat}}$ annihilates the evaluation image ideal $\mathfrak q_n(\mathcal F_n)|_{t\neq 0}$, this proves \eqref{eq:graded-terminal-vanishing}.

\emph{Case 2: Some point of $Z_n\cap Y$ lies outside $V_Y(I)$}.  We divide the components outside $V_Y(I)$ by $s_{n+1}$ and then lift the resulting section. 
Let $\varphi$ be a nonzero homogeneous element in the ideal $\sqrt I^{d-\sum_{j=1}^{n-1} \deg_L([Z_j]_{n+1-j}-[Z_j^{{B \mathrm {tf}}}]_{n+1-j})}$. 
On every local component of $Z_n\cap Y$ supported on $V_Y(I)$, the restriction of $\varphi$ is zero by a similar length bound argument in the proof of Lemma~\ref{lem:length}, because the exponent is at least $\deg_L[Z_n]_1$.
On every other component, $s_{n+1}$ is nowhere zero.
Taking zero on the former components and $\varphi/s_{n+1}$ on the latter defines a section of $\mathcal O_{Z_n\cap Y}(\deg\varphi-a_n)$. In what follows, we show that it lifts to some $\widetilde\varphi\in S_{\deg\varphi-a_n}$. So $\varphi-s_{n+1}\widetilde\varphi\in I_{Z_n\cap Y}=K^{\mathrm{sat}}$.
Both $K^{\mathrm{sat}}$ and $(s_{n+1})\subseteq I$ annihilate $\mathfrak q_n(\mathcal F_n)|_{t\ne0}$, thus $(\sqrt I)^d(S/I)=0$, as required.

Now we check the required lifting. In our case, $S/K^{\mathrm{sat}}$ is a standard graded Cohen-Macaulay ring, and $$\dim_k H^0(Z_n\cap Y,\mathcal O_{Z_n\cap Y})=\deg (S/K^{\mathrm{sat}})=\deg_L[Z_n]_1.$$  Choose a general linear nonzero divisor $h\in (S/K^{\mathrm{sat}})_1$ such that $S/K^{\mathrm{sat}}$ is finite flat over $k[h]$. Then we can write $ S/K^{\mathrm{sat}}=\oplus_i k[h](-b_i)$ as graded modules with shifted degrees $b_i$. Thus the final curve cone admits a finite flat morphism $\pi:\operatorname{Proj} S/K^{\mathrm{sat}} [t]\rightarrow\mathbb P^1_{[h:t]}$ such that $\pi_*\mathcal O_{\operatorname{Proj}S/K^{\mathrm{sat}}[t]}\cong \oplus_i\mathcal O_{\mathbb P^1}(-b_i)$.  
Since $S/K^{\mathrm{sat}}[t]$ is Cohen-Macaulay of dimension two, its graded pieces compute these global sections. For $v\geq\max_i b_i$, we have the vanishing $H^1(\mathbb P^1,\mathcal O(v-1-b_i))=0$ and therefore  $S_v\rightarrow H^0(Z_n\cap Y,\mathcal O_{Z_n\cap Y}(v))$ is surjective.  

From the surjectivity established above, verifying the liftability of $\varphi/s_{n+1}$ reduces to checking the degree inequality 
\[d-\sum_{j=1}^{n-1}\deg_L([Z_j]_{n+1-j}-[Z_j^{{B \mathrm {tf}}}]_{n+1-j})-a_n\geq \max b_i.\]

If at least one cycle difference $\deg_L([Z_j]_{n+1-j}-[Z_j^{{B \mathrm {tf}}}]_{n+1-j})$ has positive degree, every weight $a_j\cdots a_{n-1}$ in \eqref{eq:graded-degree-budget} is at least $a_n$. Since the cycle degrees are nonnegative integers and at least one is positive, that identity and the definition of $d$ give
\[
(d-\sum_{j=1}^{n-1} \deg_L([Z_j]_{n+1-j}-[Z_j^{{B \mathrm {tf}}}]_{n+1-j}))-a_n \geq \deg_L[Z_n]_1-1.
\]
Moreover, $S/K^{\mathrm{sat}}+(h)$ is a standard graded Artin algebra of length $\deg_L[Z_n]_1$. So $\deg_L[Z_n]_1-1\geq\max_i b_i$.

If all the cycle differences $\deg ([Z_j]_{n+1-j}-[Z_j^{B\mathrm {tf}}]_{n+1-j})=0$, then $Z_j^{B\mathrm{tf}}=Z_j$ for every $0\leq j<n$. 
Indeed, inductively the preceding quotients are Cohen-Macaulay, so any nonzero $|B|$-torsion would contribute a positive top-dimensional cycle.
Consequently, $s_1,\ldots,s_n$ is an $S$-regular sequence and $K=(s_1,\ldots,s_n)=K^{\mathrm{sat}}$.
A linear Artin reduction of $S$ has length $\deg Y$, hence highest nonzero degree at most $\deg Y-1$. The Hilbert-series formula for this regular sequence therefore gives
\[
d-a_n\geq\deg Y-1+(a_m-1)+\sum_{i=1}^{n-1}(a_i-1)\geq  \max_i b_i,
\]
so the required inequality follows as well.
\end{proof}

\begin{remark}
Retain the assumptions and notation of Theorem~\ref{thm:graded-homogeneous-containment}, but remove the domain assumption on $S$. The Cohen-Macaulay hypothesis still gives the weaker bound 
$$(\sqrt I)^{\deg Y\cdot a_1\cdots a_{\min\{m,n+1\}}} \subseteq I.$$ 
Indeed, on $X=\operatorname{Proj}S[t]$, take $\mathcal J=\widetilde{IS[t]}$. We apply the triangular avoidance argument of Section~\ref{sec:refined-degrees} starting with degree $a_1$. This gives successive sections of degrees $a_1,\ldots,a_{\min\{m,n+1\}}$ that are regular outside $V(\mathcal J)$. 
If $m<n+1$, these sections generate $\mathcal J$, so the condition~\textup{(2)} of Theorem~\ref{thm:unified-intersection} also holds. Applying Theorem~\ref{thm:unified-intersection} with auxiliary
line bundle $\mathcal O_X(1)$ and restricting to $D_+(t)\cong\operatorname{Spec}S$ proves the assertion.
However, this estimate does not, in general, recover the sharp exponent of \cite[Theorem~1.5]{Kollar88} for polynomial rings. 
\end{remark}

\begin{remark} The prime-power product conclusion in the second part of Theorem~\ref{thm:unified-intersection} admits a graded refinement under the hypotheses of Theorem~\ref{thm:graded-homogeneous-containment}: There exist finitely many reduced irreducible closed subschemes $W_\alpha\subseteq V_Y(I)$ with
\[
\mathrm {codim}_Y W_\alpha\in\{1,\ldots,\min\{m,n\}\},
\]
positive integers $e_\alpha$, and an integer $e_0\geq0$ such that $S_+^{e_0}\prod_\alpha I_{W_\alpha}^{e_\alpha}\subseteq I$ and $\bigcup_\alpha W_\alpha=V_Y(I)$. Moreover, we have $e_0+\sum_\alpha e_\alpha\deg W_\alpha\leq d$ and 
$$\textstyle\sum_\alpha e_\alpha\left(\textstyle\prod_{i=\mathrm {codim}_Y W_\alpha}^{\min\{m,n\}-1}a_i\right)\deg W_\alpha\leq \deg Y\cdot a_m a_1\cdots a_{\min\{m,n\}-1}.$$
If $m\leq n$, one may take $e_0=0$. This refinement is obtained by retaining the generic multiplicities of the cokernels $\mathcal Q_j$ arising from the $|B|$-torsion removal steps and introducing the factor $S_+^{e_0}$ in the final step.
\end{remark}

We next give the proof of Corollary~\ref{cor:graded-affine}.

\begin{proof}[Proof of Corollary~\ref{cor:graded-affine}]
Choose the homogenizing coordinate $x_0$ so that
\[
V=D_+(x_0)\subseteq\operatorname{Proj}S,
\]
and homogenize the $f_i$ to $F_i\in S_{a_i}$. Set $I=(F_1,\ldots,F_m)$.
The first assertion follows by applying Corollary~\ref{thm:mixed-degrees} on $Y$ to $\widetilde I$, and then restricting the resulting containment to $V$.
If $\mathfrak a=k[V]$, homogenizing a representation of $1$ shows that $x_0\in\sqrt I$. Theorem~\ref{thm:graded-homogeneous-containment} gives $x_0^d=\sum_{i=1}^m G_iF_i$ with $G_i\in S_{d-a_i}$. Dehomogenization proves \eqref{eq:graded-affine-certificate}. The last assertion follows from the Nullstellensatz and faithful-flat descent.
\end{proof}

The following example shows that Theorem~\ref{thm:graded-homogeneous-containment} and Corollary~\ref{cor:graded-affine} are optimal.
\begin{example}
\label{ex:graded-radical-sharpness}

Fix $n,m,\delta\geq1$ and $a_1\geq\cdots\geq a_m\geq1$. Both rings below are Cohen-Macaulay hypersurface domains of dimension $n+1$ and degree $\delta$.
\begin{enumerate}
    \item \emph{The product bound.} Let $r=\min\{m,n\}$ and take $$S=k[t,z,x_1,\ldots,x_n]/(x_rt^{\delta-1}-z^\delta), \quad I=(x_1^{a_m},\ x_jt^{a_j-1}-x_{j+1}^{a_j}\ (1\leq j<r)).$$ Add the redundant generators $x_1^{a_j}$ for $r\leq j<m$ to obtain all prescribed degrees. As in Example~\ref{ex:ambient-sharp}, successive elimination gives $z\in\sqrt I$ and
    \[
    S/(I,t-1)\cong k[z,x_{r+1},\ldots,x_n]/(z^{\delta a_m a_1\cdots a_{r-1}}).
    \]
    Thus $t^v z^{({\delta a_m a_1\cdots a_{r-1}})-1}\notin I$ for every $v\geq0$. Since dehomogenization at $t=1$ preserves the generator degrees, ${\delta a_m a_1\cdots a_{r-1}}$ is optimal for $\sqrt{I}^d\subseteq I^\mathrm {sat}$, $(\sqrt{I}^d)_v\subseteq I_v$ $v\geq v_0$, and the affine radical-power bound $\sqrt{\mathfrak a}^d\subseteq \mathfrak a$.

\item \emph{The additive bound.} For $m>n$, take $$S=k[u_0,\ldots,u_n,w]/(w^\delta-u_0u_1^{\delta-1}), \quad  I=(u_0^{a_m},u_1^{a_1},\ldots,u_n^{a_n}),$$ together with the redundant generators $u_0^{a_j}$ for $n<j<m$. The displayed parameters form a regular sequence, so
\[ \operatorname{Hilb}_{S/I}(z)  =(1+\cdots+z^{\delta-1})(1+\cdots+z^{a_m-1})   \prod_{i=1}^n(1+\cdots+z^{a_i-1}).
\]
Since the last nonzero degree is $\delta-1+a_m-1+\sum_{i=1}^n(a_i-1)$, the least exponent is $\delta-1+a_m+\sum_{i=1}^n(a_i-1)$. Together with (1), this proves optimality of $d=\max\{\delta a_m a_1\cdots a_{r-1},\delta-1+a_m+\sum_{i=1}^n(a_i-1)\}$.

    \item \emph{Degree bound for $g_i$'s.} For $0\ne h\in S_1$, we let $\bar f_i=f_i/h^{a_i}$. By the homogenization in Corollary~\ref{cor:graded-affine}, there exist polynomial functions $g_i$ with $\deg g_i+a_i\leq v$ and satisfying $ 1=\sum_{i=1}^m g_i\bar f_i$, if and only if $h^v\in I$.
\begin{itemize}
    \item In (1), take $h=z$. If $\delta>1$ or $r>1$, all chosen generators are nonzero modulo $z$, so their affine degrees remain $a_i$. We have $z^{\delta a_m a_1\cdots a_{r-1}-1}\notin I$, while Theorem~\ref{thm:graded-homogeneous-containment}, applied to the $r$ generators gives $z^{\delta a_m a_1\cdots a_{r-1}}\in I$. Thus the least possible value of $\max_{g_i}(\deg g_i+a_i)$ among such representations is $\delta a_m a_1\cdots a_{r-1}$.

    \item In (2), assume characteristic zero. Polarization shows that powers of linear forms span $(S/I)_{\delta-1+a_m-1+\sum_{i=1}^n(a_i-1)}$. So a general element $h$ in $S_1$ satisfies $h^{\delta-1+a_m-1+\sum_{i=1}^n(a_i-1)}\notin I$ and $h^{\delta-1+a_m+\sum_{i=1}^n(a_i-1)}\in I$. We may also choose $h$ regular on the quotient $S/(f_i)$, since these quotients are Cohen-Macaulay of positive dimension. This preserves the degrees $a_i$ after dehomogenization. Consequently, the least possible value of $\max_{g_i}(\deg g_i+a_i)$ is $\delta-1+a_m+\sum_{i=1}^n(a_i-1)$. If $\delta=r=1<m$, then $n=1$ and $\delta-1+a_m+\sum_{i=1}^n(a_i-1)\geq \delta a_m a_1\cdots a_{r-1}$. So the second family suffices. 
    
    \item In the remaining case $\delta=m=1$, a nonconstant polynomial function on a degree-one affine variety cannot generate the unit ideal. 
\end{itemize}
Hence the uniform bound $\deg g_i+a_i\leq d$ is optimal.
\end{enumerate}
\end{example}

\begin{remark}
The examples in parts (1) and (2) establish the sharpness of the containment bounds. These work over every field. In part (3), characteristic zero is used only to attain the additive bound for $\max_{g_i}(\deg g_i+a_i)$ by a power of a linear form. This characteristic-zero construction also shows that the uniform bound $\deg g_i+a_i\leq d$ cannot be improved in a statement valid over arbitrary fields. For $\delta=1$, the affine charts in the examples of part (3) are affine spaces, so these examples also establish sharpness for $\max_{g_i}\deg(g_if_i)$.
\end{remark}

\vspace{20pt}
\noindent \textbf{Acknowledgment}. 
While exploring the construction of pure sheaves via local duality, we consulted AI for the related properties of extension sheaves. The AI introduced us to the Koszul–Fitting ideals and provided the reference paper \cite{BH}. We also used AI for language suggestions. The authors take full responsibility for the paper.

\vspace{20pt}

\noindent \textbf{Statements and Declarations}.

\noindent\emph{Competing interests}. The authors have no relevant financial or non-financial interests to disclose.

\noindent\emph{Funding}. This work is supported partially by the National Key R\&D Program of China (No. 2021YFA1002600). Z. Li is partially supported by NSFC (No. 12201060). X. Meng is partially supported by NSFC (No. 12271057). X. Su is partially supported by NSFC (No. 12301056) and the Fundamental Research Funds for the Central Universities. L. Zhu is partially supported by NSFC (No. 12371080).

\end{document}